\documentclass{article} 
\usepackage{amsmath}
\usepackage{amssymb}
\usepackage{amsthm}
\usepackage{enumerate}
\usepackage{color}
\usepackage[english]{babel}
\usepackage[latin1]{inputenc} 

\newtheorem{theorem}{Theorem}[section]
\newtheorem{corollary}[theorem]{Corollary}
\newtheorem{lemma}[theorem]{Lemma}
\newtheorem{proposition}[theorem]{Proposition}
\newtheorem{remark}[theorem]{Remark}
\newtheorem{example}[theorem]{Example}
\newtheorem{definition}[theorem]{Definition}

\newcommand{\C}{{\ensuremath{\mathbb{C}}}}
\newcommand{\Cm}{{\ensuremath{\C^{m\times n}}}}
\newcommand{\Cn}{{\ensuremath{\C^{n\times m}}}}
\newcommand{\Cnn}{{\ensuremath{\C^{n\times n}}}}
\newcommand{\Cmm}{{\ensuremath{\C^{m\times m}}}}
\newcommand{\Ra}{{\ensuremath{\cal R}}}
\newcommand{\Nu}{{\ensuremath{\cal N}}}
\newcommand{\ra}{{\ensuremath{\text{\rm rank}}}}
\newcommand{\ind}{{\ensuremath{\text{\rm Ind}}}}

\newcommand{\odagger}{\mathrel{\text{\textcircled{$\dagger$}}}}
\newcommand{\core}{\mathrel{\text{\textcircled{$\#$}}}}

\usepackage[left=3cm,top=3.5cm,right=3cm,bottom=3.5cm]{geometry}

\begin{document}

\title{Parametrizing $W$-weighted BT inverse to obtain the $W$-weighted $q$-BT inverse}

\author{D.E. Ferreyra\thanks{Universidad Nacional de R\'io Cuarto, CONICET, FCEFQyN, RN 36 KM 601, 5800 R\'io Cuarto, C\'ordoba, Argentina. E-mail: \texttt{deferreyra@exa.unrc.edu.ar}},\, 
N. Thome\thanks{Instituto Universitario de
Matem\'atica Multidisciplinar, Universitat Polit\`ecnica de
Val\`encia, 46022 Valencia, Spain. E-mail: \texttt{njthome@mat.upv.es}},\,
C. Torigino\thanks{Universidad Aut\'onoma de Entre R\'{\i}os, FCyT, 25 de Mayo 385, 3260 Concepci\'on del Uruguay, Entre R\'ios, Argentina. E-mail: \texttt{torigino.carlos@uader.edu.ar}}
}

\date{}
\maketitle

\begin{abstract}
The core-EP and BT inverses for rectangular matrices were studied recently in the literature. The main aim of this paper is to unify both concepts by means of a new kind of generalized inverse called $W$-weighted $q$-BT inverse. We analyze its existence and uniqueness by considering an adequate 
matrix system. Basic properties and
some interesting characterizations are proved for this new weighted generalized inverse. Also, we give a canonical form  of the $W$-weighted $q$-BT inverse by means of the weighted core-EP decomposition.
\end{abstract}



\section{Introduction and preliminaries}

We denote by $\mathbb{C}^{m\times n}$  the set of all $m\times n$ complex matrices. Let $A\in \mathbb{C}^{m\times n}$. The conjugate transpose, rank, null space and column space of $A$ are denoted by $A^*$, $\ra(A)$, $\Nu(A)$, and $\Ra(A)$, respectively. The index of $A\in \Cnn$, denoted by $\ind(A)$, is the smallest nonnegative integer $k$ such that $\ra(A^k) = \ra(A^{k+1})$. Moreover, $A^0=I_n$ will refer to the $n \times n$ identity matrix, and $0$ will denote the null matrix of appropriate size. The standard notations $P_S$ and $P_{S,T}$ stand for the orthogonal projector onto a subspace $S$ and a projector onto $S$ along $T$, respectively, when $\C^n$ is equal to the direct sum of subspaces $S$ and $T$.

The Drazin inverse of a matrix $A\in \Cnn$ is the unique matrix $X=A^d \in \Cnn$ that satisfies 
\[X A^{k+1} =A^k, \quad XAX=X, \quad AX=XA, ~\text{where} ~k=\ind(A).\]
When $\ind(A)=1$,  the Drazin inverse is called the group inverse of $A$ and is denoted by $A^\#$. 

The Moore-Penrose inverse of a matrix $A\in \Cm$ is the unique matrix $X=A^\dag \in \Cn$ that satisfies the Penrose equations
\[AXA=A, \quad XAX=X,
\quad (AX)^*=AX, \quad(XA)^*=XA.\]

We will denote by  $P_A$  the orthogonal projector $AA^{\dag}$ onto the subspace $\Ra(A)$.

In 2014, Manjunatha Prassad and Mohana  \cite{PrMo}  introduced the core-EP inverse of a matrix $A\in \Cnn$ of index $k$ as the unique matrix  $X=A^{\odagger}\in \Cnn$ that satisfies the conditions $XAX=X$ and $\Ra(X)=\Ra(X^*)=\Ra(A^k)$. That same year, Baksalary and Trenkler \cite{BaTr2}  defined the BT inverse of $A$ as the matrix $A^{\diamond}=(AP_A)^\dag$. When the matrix $A$ has index 1, both inverses are reduced to the well-known core inverse $A^{\core}=A^\#AA^\dag$ of $A$ \cite{BaTr}. 
 
In 1980, Cline and Greville \cite{ClGr} extended the Drazin inverse to rectangular matrices and it was called the $W$-weighted Drazin inverse.
Let $W\in \Cn$ be a fixed nonzero matrix. We recall that the $W$-weighted Drazin inverse of $A\in \Cm$, denoted by  $A^{d,W}$, is the unique matrix 
$X\in \Cm$ satisfying the three equations 
\[XWAWX=X, \quad AWX=XWA, \quad   XW(AW)^{k+1}=(AW)^k,\]
where $k=\max\{\ind(AW),\ind(WA)\}$. \\
If $k=1$, the $W$-weighted Drazin inverse of $A$ is called the $W$-weighted group inverse of $A$ and is denoted by $A^{\#,W}$. When $m=n$ and $W=I_n$, we recover the Drazin inverse, that is,  $A^{d,I_n}=A^d$. 

The $W$-weighted Drazin inverse satisfies the following  two dual  representations
\begin{equation}\label{properties w drazin}
A^{d,W}=A[(WA)^d]^2=[(AW)^d]^2A,\quad \text{whence} \quad A^{d,W}W=(AW)^d,  \quad WA^{d,W}=(WA)^d.
\end{equation}
Interesting representations and properties of the $W$-weighted Drazin
inverse were studied in \cite{StKaMa}.

Similarly, the core-EP inverse was extended to rectangular matrices in \cite{FeLeTh1}. It was named $W$-weighted core-EP inverse, and defined as $A^{\odagger,W}=(WAWP_{(AW)^k})^\dag$, which is the unique solution of 
\begin{equation}\label{w cep system}
WAWX=P_{(WA)^k}, \quad \Ra(X)\subseteq \Ra((AW)^k).
\end{equation}

For the particular  case $k=1$, the $W$-weighted core-EP inverse of $A$ is known as the $W$-weighted core inverse of $A$ and  denoted by $A^{\core,W}$. Clearly, when $m=n$ and $W=I_n$, we recover the core-EP inverse, that is,  $A^{\odagger,I_n}=A^{\odagger}$. 

The $W$-weighted core-EP inverse satisfies the following interesting properties \cite{FeLeTh1, GaChPa}
\begin{equation}\label{properties w core ep}
A^{\odagger,W}=A[(WA)^{\odagger}]^2, \quad A^{\odagger,W}WP_{(AW)^k}=(AW)^{\odagger},  \quad P_{(WA)^k}WA^{\odagger,W}=(WA)^{\odagger}.
\end{equation}

Recently, the $W$-weighted BT inverse of $A$ was defined in \cite{FeThTo} as the unique matrix $X=A^{\diamond, W} \in \Cm$ satisfying the following equations
\begin{equation}\label{def weighted bt}
XWAWX=X, \quad XWA=[W(AW)^2(AW)^\dag]^\dag WA, \quad AWX=AW[(WA)^2W(AW)^\dag]^\dag.
\end{equation}
It was also established that $A^{\diamond, W}=(WAWP_{AW})^\dag$.

Interesting results including different kinds of weighted generalized inverses can be found in \cite{Meng, Mo1, MoKo}.

In this paper we unify the definitions given in \eqref{w cep system} and \eqref{def weighted bt} given rise a new kind of generalized inverse called $W$-weighted $q$-BT inverse. We analyze its existence and uniqueness by considering an adequate matrix system. 

This paper is organized as follows. In Section 2, we present results of existence and uniqueness of the $W$-weighted $q$-BT inverse. More precisely, the existence will be characterized as the unique solution of three matrix equations. 
In Section 3, we obtain some characterizations of the $W$-weighted $q$-BT inverse. As an interesting consequence, we present new characterizations of the $W$-weighted core-EP 
and $W$-weighted BT inverses. In Section 4, we obtain a canonical form of the $W$-weighted $q$-BT inverse  by using a simultaneous decomposition of the matrices $A$ and $W$ called the weighted core-EP decomposition. Finally, some more properties of this new generalized inverse are investigated.

\section{Existence and Uniqueness}

In this section, we define and investigate the $W$-weighted $q$-BT  inverse for rectangular matrices $A\in \Cm$ by considering a non-null weight $W\in \Cn$. 

We start with a result of existence and uniqueness. Before that, we need the following auxiliary lemma.

\begin{lemma}\label{lemma 1} Let $A\in \Cm$ and $B\in \C^{n\times s}$. Then $P_B(AP_B)^\dag=(AP_B)^\dag$.
\end{lemma}
\begin{proof}
Since $(I_n-P_B)P_B A^*=0$ trivially holds, we have that $\Ra((AP_B)^\dag)\subseteq \Nu((I_n-P_B))$ is always valid, which in turn is equivalent to $P_B(AP_B)^\dag=(AP_B)^\dag$.
\end{proof}

\begin{theorem}\label{existence uniqueness}
Let $A\in \Cm$, $0\neq W\in \Cn$ and $q\in \mathbb{N} \cup \{0\}$. The system of equations
\begin{equation}\label{system wqbt}
XWAWX=X, \quad XWA=(WAWP_{(AW)^q})^\dag WA, \quad AWX=AW(WAWP_{(AW)^q})^\dag, 
\end{equation}  is  consistent and has a unique solution 
$X=(WAWP_{(AW)^q})^\dag$.
\end{theorem}
\begin{proof}
\emph{Existence.}
Let $X:=(WAWP_{(AW)^q})^\dag$. Clearly, $X$ satisfies the two last equations in \eqref{system wqbt}. Moreover, from Lemma \ref{lemma 1} we have 
\begin{eqnarray*}
XWAWX &=& (WAWP_{(AW)^q})^\dag WAW(WAWP_{(AW)^q})^\dag\\
&=& (WAWP_{(AW)^q})^\dag WAW P_{(AW)^q}(WAWP_{(AW)^q})^\dag \\
&=& (WAWP_{(AW)^q})^\dag \\
&=&  X.
\end{eqnarray*}
Thus, $X$ is a solution to \eqref{system wqbt}.\\
\emph{Uniqueness.}
Any arbitrary solution $X$ to the system \eqref{system wqbt} satisfies
\begin{eqnarray*}
X &=&(XWA)WX \\
&=& (WAWP_{(AW)^q})^\dag W(AWX) \\
&=& (WAWP_{(AW)^q})^\dag WAW(WAWP_{(AW)^q})^\dag\\
&=& (WAWP_{(AW)^q})^\dag WAW P_{(AW)^q}(WAWP_{(AW)^q})^\dag \\
&=& (WAWP_{(AW)^q})^\dag,
\end{eqnarray*}
which implies that the matrix $X=(WAWP_{(AW)^q})^\dag$ is the unique solution to 
\eqref{system wqbt}. 
\end{proof}

\begin{example} We consider the matrices 
\[
A=\left[\begin{array}{cccc}
1 &  1 & 0 & 0\\
0 &  0 & 1 & 0 \\
0 &  0 & 1 & 1\\
0 &  0 & 0 & 1 \\
0 &  0 & 0 & 0
\end{array}\right] \quad \text{ and }  \quad W=\left[\begin{array}{ccccc}
1 &  0 & 1 & 0 & 0 \\
0 &  0 & 1 & 1 & 0 \\
0 &  0 & 0 & 1 & 1 \\
0 &  0 & 0 & 0 & 1
\end{array}\right].
\]
It is easy to see that $k=\max\{{\rm Ind}(AW), {\rm Ind}(WA)\}=\max\{3,3\}=3$. \\ 
Let $X_0:=(WAWP_{(AW)^q})^\dag$, $X:=Q_{AW}X_0+(I_m-Q_{AW})W^*$ and we set $q=1$. By Lemma \ref{lemma 1}, it is clear that $X_0WAWX_0=X_0$. Then,
\begin{eqnarray*}
XWAWX &=& [Q_{AW}X_0+(I_m-Q_{AW})W^*]WAW[Q_{AW}X_0+(I_m-Q_{AW})W^*] \\
&=& [Q_{AW}X_0+(I_m-Q_{AW})W^*]WAWX_0\\
&=& Q_{AW}X_0WAWX_0+(I_m-Q_{AW})W^*WAWX_0 \\
&=& Q_{AW}X_0+(I_m-Q_{AW})W^*WAWX_0\\
&=& Q_{AW}X_0+(I_m-Q_{AW})W^*\\
&=& X
\end{eqnarray*}
and
\begin{eqnarray*}
AWX &=& AW[Q_{AW}X_0+(I_m-Q_{AW})W^*]\\
&=& AWQ_{AW}X_0\\
&=& AWX_0.
\end{eqnarray*} 
We obtain that $X:=Q_{AW}X_0+(I_m-Q_{AW})W^*$ satisfies the first and third equations in \eqref{system wqbt}. However, we can observe that the second equation does not hold. In fact, 
\[XWA=\left[\begin{array}{ccrr}
\frac{3}{5} &  \frac{3}{5} & -\frac{1}{5} & -1\\
0 &  0 & 0 & 0 \\
\frac{1}{5} &  \frac{1}{5} & -\frac{2}{5} & -1\\
0 &  0 & 1 & 2 \\
0 &  0 & 0 & 0
\end{array}\right] \neq \left[\begin{array}{ccrr}
\frac{1}{3} &  \frac{1}{3} & -\frac{2}{3} & -\frac{5}{3}\\
\frac{1}{3} &  \frac{1}{3} & -\frac{7}{6} & -\frac{8}{3} \\
\frac{1}{3} &  \frac{1}{3} & -\frac{1}{6} & -\frac{2}{3}\\
0 &  0 & 1 & 2 \\
0 &  0 & 0 & 0
\end{array}\right]=X_0WA.
\]

\end{example}

\begin{definition}\label{def wqbt} 
Let $A \in \Cm$, $0\neq W\in \Cn$, $k=\max\{\ind(AW), \ind(WA)\}$, and $q\in \mathbb{N}\cup \{0\}$. The unique matrix  $X \in\Cm$ that satisfies the system 
 (\ref{system wqbt})
is called the $W$-weighted $q$-BT inverse of $A$, and is denoted by $A^{\diamond_q,W}$. 
\end{definition}

\begin{remark} \label{rem q-BT} Note that when $m=n$ and $W=I_n$, the $W$-weighted $q$-BT inverse of $A$ gives rise a new generalized inverse for square matrices. For simplicity, it will be denoted as  $A^{\diamond_q}:=(AP_{A^q})^\dag$ and will be called the $q$-BT inverse of $A$. 
\end{remark}

The motivation for the study of this new kind of generalized inverse  is stated in the following result by showing that it extends certain inverses known in the literature.

\begin{corollary} \label{corolario 1}
Let $A \in \Cm$, $0\neq W\in \Cn$, $k=\max\{\ind(AW), \ind(WA)\}$, and $q\in \mathbb{N}\cup \{0\}$. Then 
\begin{enumerate}[(i)]
\item $A^{\diamond_q,W}=(WAW)^{\dag}$ if $q=0$; 
\item $A^{\diamond_q,W}=A^{\diamond,W}$ if $q=1$;
\item $A^{\diamond_q,W}=A^{\odagger,W}$ if either $q=\ind(AW)$ or $q\ge k$. 
\end{enumerate}
\end{corollary}
\begin{proof} (i) Follows from Theorem \ref{existence uniqueness} with $q=0$. \\
(ii) It is a consequence from Theorem \ref{existence uniqueness} and the expression of $A^{\diamond,W}$ recalled below \eqref{def weighted bt}.\\
(iii) It follows from Theorem \ref{existence uniqueness}, the definition of $A^{\odagger,W}$ and the fact that $P_{(AW)^q}=P_{(AW)^k}$ when either $q=\ind(AW)$ or $q\ge k$.
\end{proof}

\begin{remark} When $WAW=A$, from the above corollary it follows that the $W$-weighted $q$-BT inverse of $A$ reduces to the Moore-Penrose inverse of $A$. Note that the condition $WAW=A$ is a Stein equation (in $A$). We recall that this equation has important applications in system theory, among them, the stability analysis of discrete-time systems \cite{Fu}.
\end{remark}
\begin{remark}  If $m=n$ and $W=I_n$, from Corollary \ref{corolario 1} we deduce that the $W$-weighted $q$-BT inverse  concides with the BT inverse and core-EP inverse,  when  $q=1$ and $q\ge k=\ind(A)$, respectively. 
\end{remark}

 An interesting relationship between the products $AW$ and  $WA$ is
\begin{equation}\label{w-product}
(AW)^{\ell-1} A= A(WA)^{\ell-1}, \quad \ell \in \mathbb{N}.
\end{equation}

\begin{corollary}\label{representacion q-BT} Let $A\in \Cm$, $0\neq W\in \Cn$ and $q\in \mathbb{N} \cup \{0\}$. Then 
\[A^{\diamond_q,W}=[W(AW)^{(q+1)} [(AW)^q]^\dag]^\dag=[(WA)^{q+1} W[(AW)^q]^\dag]^\dag.\]
\end{corollary}
\begin{proof}
Follows from Theorem \ref{existence uniqueness} and \eqref{w-product}.
\end{proof}

In the following example we show that when $1<q<k$ (eventually with $q \neq \ind(AW)$), this new inverse is different from other known ones. 

\begin{example} 
Consider the matrices 
\[
A=\left[\begin{array}{ccc}
1 &  1 & 0 \\
0 &  1 & 0 \\
0 &  0 & 1 \\
0 &  0 & 0
\end{array}\right] \quad \text{and}  \quad W=\left[\begin{array}{cccc}
1 &  1 & 0 & 0 \\
0 &  0 & 1 & 0 \\
0 &  0 & 0 & 1
\end{array}\right].\]
Since $\ind(AW)=3$ and $\ind(WA)=2$, we have $k=\max\{\ind(AW), \ind(WA)\}=3$. Therefore, we must consider $q=2$. Thus, the $W$-weighted core-EP inverse, the $W$-weighted BT inverse, and the $W$-weighted $2$-BT inverse are given by 
\[ 
A^{\odagger,W} =
\left[\begin{array}{ccc}
1 &  0 & 0 \\
0 &  0 & 0 \\
0 &  0 & 0 \\
0 &  0 & 0
\end{array}\right], \quad 
A^{\diamond,W}=
\left[\begin{array}{ccc}
\frac{1}{6} &  0 & 0 \\
\frac{1}{6} &  0 & 0 \\
\frac{1}{3} &  0 & 0 \\
0 &  0 & 0
\end{array}\right], \quad A^{\diamond_2,W}=
\left[\begin{array}{ccc}
\frac{1}{2} &  0 & 0 \\
\frac{1}{2} &  0 & 0 \\
0 &  0 & 0 \\
0 &  0 & 0
\end{array}\right]. \] 
\end{example} 

Some properties of the $W$-weighted $q$-BT inverse are established below. For example,  the $W$-weighted $q$-BT inverse can be expressed in terms of the $q$-BT inverse. In particular, the $q$-BT inverse provides the range and null space of the $W$-weighted BT inverse.

\begin{theorem}\label{properties BT}
Let $A\in \Cm$, $0\neq W\in \Cn$ and $q\in \mathbb{N} \cup \{0\}$. Then the following statements hold:
\begin{enumerate}[{\rm (i)}]
\item $A^{\diamond_q, W}=(W[(AW)^{\diamond_q}]^\dag)^\dag$.
\item $\Ra(A^{\diamond_q,W})=\Ra(P_{(AW)^q}(WAW)^*)$ and $\Nu(A^{\diamond_q,W})=\Nu(P_{(AW)^q}(WAW)^*)$.
\item $\Ra(A^{\diamond, W})= \Ra([(AW)^{\diamond_q}]^\dag)^* W^*)$ and $\Nu(A^{\diamond,W})=\Nu([(AW)^{\diamond_q}]^\dag)^* W^*)$. 
\item $\Ra(A^{\diamond_q,W})=\Ra([((AW)^q)^\dag]^*[(AW)^{q+1}]^* W^*)$  and $\Nu(A^{\diamond_q,W})=\Nu([(AW)^{q+1}]^* W^*)$.
\item $\Ra(A^{\diamond_q,W})\subseteq \Ra((AW)^q)$.
\item $P_{(AW)^q} A^{\diamond_q,W}=A^{\diamond_q,W}$.
\end{enumerate}
\end{theorem}
\begin{proof}
(i) By Theorem \ref{existence uniqueness} we have $A^{\diamond_q,W}=(W[AWP_{(AW)^q}])^\dag$. Now, by Remark \ref{rem q-BT}  we deduce  $AWP_{(AW)^q}=((AW)^{\diamond_q})^\dag$, whence the  statement  is clear.\\
(ii) By Theorem \ref{existence uniqueness} we have $A^{\diamond_q,W}=(WAWP_{(AW)^q})^\dag$. Now, the statement follows of the properties $\Ra(B^\dag)=\Ra(B^*)$ and $\Nu(B^\dag)=\Nu(B^*)$. \\
(iii) It follows immediately  from  part  (i). \\
(iv) By Corollary \ref{representacion q-BT} and the property $\Ra(B^\dag)=\Ra(B^*)$ we get 
\[\Ra(A^{\diamond_q,W})=\Ra([W(AW)^{(q+1)} ((AW)^q)^\dag]^*)=\Ra([((AW)^q)^\dag]^*[(AW)^{q+1}]^* W^*).\]  
Similarly, Corollary \ref{representacion q-BT} and the property $\Nu(B^\dag)=\Nu(B^*)$ imply
\begin{eqnarray*}
\Nu(A^{\diamond_q,W}) &=& \Nu([((AW)^q)^\dag]^*[(AW)^{q+1}]^* W^*)\\
&=& \Nu([((AW)^q)^\dag]^*[(AW)^q]^* (AW)^* W^*) \\
&\subseteq & \Nu([(AW)^q]^*[((AW)^q)^\dag]^*[(AW)^q]^* (AW)^* W^*)\\ 
&=& \Nu([(AW)^{q+1}]^* W^*)\\
&\subseteq & \Nu([((AW)^q)^\dag]^*[(AW)^{q+1}]^* W^*)\\
&=& \Nu(A^{\diamond_q,W}).
\end{eqnarray*}
Thus, $\Nu(A^{\diamond,W})=\Nu([(AW)^{q+1}]^* W^*)$. \\
(v) It directly follows from (ii) and the fact that $\Ra(P_{(AW)^k})=\Ra((AW)^k)$. \\
(vi) It is sufficient to note that $P_{(AW)^q} A^{\diamond_q,W}=A^{\diamond_q,W}$ holds if and only if $\Ra(A^{\diamond_q,W})\subseteq \Nu(I_m-P_{(AW)^q})=\Ra(P_{(AW)^q})=\Ra((AW)^q)$, which is true due to part (v).\\
\end{proof}
We finish this section by  showing that the $W$-weighted $q$-BT inverse can be written as  a generalized inverse with prescribed range and null space. Moreover,  some idempotent matrices related to the $W$-weighted $q$-BT inverse are found.

\begin{proposition}\label{prop 1} Let $A\in \Cm$, $0\neq W\in \Cn$ and $q\in \mathbb{N} \cup \{0\}$. Then the following representations are valid:
\begin{enumerate}[(i)]
\item $A^{\diamond_q,W}=(WAW)^{(2)}_{\Ra(P_{(AW)^q}(WAW)^*),\,\Nu([(AW)^{q+1}]^* W^*)}$;
\item $WAWA^{\diamond_q,W}=P_{\Ra(W [(AW)^{\diamond_q}]^\dag (WAW)^*),\,\Nu([(AW)^{q+1}]^* W^*)}$;
\item $A^{\diamond_q,W}WAW=P_{\Ra(P_{(AW)^q}(WAW)^*),\,\Nu([(AW)^{q+1}]^* W^* WAW)}$.
\end{enumerate}
\end{proposition}
\begin{proof}
(i) By definition of the $W$-weighted $q$-BT inverse we know that $A^{\diamond, W}WAWA^{\diamond, W}=A^{\diamond, W}$. Now, parts (ii) and (iv) of Theorem \ref{properties BT} imply  $\Ra(A^{\diamond_q,W})=\Ra(P_{(AW)^q}(WAW)^*)$ and  $\Nu(A^{\diamond, W})=\Nu([(AW)^{q+1}]^* W^*)$, respectively. Thus, the statement follows by definition of an outer inverse with  prescribed range and null space.
\\
(ii) Since $A^{\diamond, W}WAWA^{\diamond, W}=A^{\diamond, W}$ by definition, we have that  $WAWA^{\diamond_q, W}$ is idempotent.  Also, from Theorem \ref{properties BT} (ii) we obtain 
\begin{eqnarray*} 
\Ra(WAWA^{\diamond_q, W})&=& WAW\Ra(A^{\diamond_q, W})\\
&=& WAW\Ra(P_{(AW)^q}(WAW)^*) \\
&=& W\Ra(AWP_{(AW)^q}(WAW)^*)\\
&=& W\Ra([(AW)^{\diamond_q}]^\dag (WAW)^*)\\
&=& \Ra(W[(AW)^{\diamond_q}]^\dag (WAW)^*).
\end{eqnarray*}
On the other hand, note that $\Nu(WAWA^{\diamond, W})=\Nu(A^{\diamond, W})$ because $A^{\diamond, W}$ is an outer inverse of $WAW$. Thus, from Theorem \ref{properties BT} (iv) we have $\Nu(WAWA^{\diamond, W})=\Nu([(AW)^{q+1}]^* W^*)$. \\
(iii) By  Theorem \ref{properties BT} (ii) we know that $\Ra(A^{\diamond_q,W})=\Ra(P_{(AW)^q}(WAW)^*)$. Thus, 
as $A^{\diamond, W}WAWA^{\diamond, W}=A^{\diamond, W}$, clearly $\Ra(A^{\diamond_q,W}WAW)=\Ra(A^{\diamond_q,W})=\Ra(P_{(AW)^q}(WAW)^*)$. \\
Similarly, from Theorem \ref{properties BT} (iv) we know that $\Nu(A^{\diamond_q,W})=\Nu([(AW)^{q+1}]^* W^*)$. On the other hand, it is easy to see that $\Nu(B)=\Nu(C)$ implies $\Nu(BD)=\Nu(CD)$, where $B$, $C$, and $D$ are complex rectangular matrices of adequate sizes.
Therefore, $\Nu(A^{\diamond,W}WAW)=\Nu([(AW)^{q+1}]^* W^* WAW)$.
\end{proof}

Recall that the Moore-Penroe inverse \cite{BeGr}, the core-EP inverse \cite[Theorem 3.2]{FeLeTh3} and the BT inverse \cite[Theorem 4.7]{FeMa4} of a matrix  $A\in \Cnn$ of index $k$, are outer inverses that can be represented as outer inverse with prescribed range and null spaces as:
\begin{equation}\label{2 inversas}
A^\dag=A^{(2)}_{\Ra(A^*),\,\Nu(A^*)},\quad A^{\odagger}=A^{(2)}_{\Ra(A^k),\,\Nu((A^k)^*)} \quad \text{and} \quad A^\diamond=A^{(2)}_{\Ra(P_A A^*),\,\Nu((A^2)^*)}.
\end{equation}

Our next theorem shows that the representations given in \eqref{2 inversas} are particular cases of the following expression for the $q$-BT inverse. 

\begin{corollary}\label{corolario outer}Let $A\in \Cnn$ and $q\in \mathbb{N} \cup \{0\}$. Then the following statements hold:
\begin{enumerate}[\rm (i)]
\item $A^{\diamond_q}=A^{(2)}_{\Ra(P_{A^q}A^*),\, \Nu([A^{q+1}]^*)}$
\item $AA^{\diamond_q}=P_{\Ra([A^{\diamond_q}]^\dag A^*),\, \Nu([A^{q+1}]^*)}$.
\item $A^{\diamond_q}A=P_{\Ra(P_{A^q}A^*),\, \Nu([A^{q+1}]^*A)}$.
\end{enumerate} 
\end{corollary}
\begin{proof}
Items (i)-(iii) immediately follow from Proposition \ref{prop 1} by taking $m=n$ and $W=I_n$. 
\end{proof}

\begin{remark} From Corollary \ref{corolario outer} (i), it is clear that when $q=0$ and $q=1$, we recover the expressions given in \eqref{2 inversas} for the Moore-Penrose inverse and the BT inverse, respectively. \\ On the other hand, if $q\ge k=\ind(A)$ we have that $\Ra(P_{A^q}A^*)=\Ra((AP_{A^q})^*)=\Ra((AP_{A^k})^\dag)=\Ra(A^{\odagger})=\Ra(A^k)$. Also, by definition of index, we obtain $\Nu((A^{q+1})^*)=\Nu((A^{k+1})^*)=\Nu((A^k)^*)$. In consequence, 
 $A^{\diamond_q}=A^{(2)}_{\Ra(A^k),\, \Nu((A^k)^*)}=A^{\odagger}$.
\end{remark}

\section{Algebraic characterizations}

In this section we give some algebraic characterizations of the $W$-weighted $q$-BT inverse.

\begin{theorem} Let $A \in \Cm$, $0\neq W\in \Cn$, $k=\max\{\ind(AW), \ind(WA)\}$, and $q\in \mathbb{N}\cup \{0\}$. 
There exists a unique matrix $X$ satisfying the conditions
\begin{equation}\label{system 1}
P_{(AW)^q}X=(WAWP_{(AW)^q})^\dag \quad \text{and} \quad \Ra(X)\subseteq\Ra((AW)^q)
\end{equation}
and is given by $X=A^{\diamond_q,W}$.
\end{theorem}
\begin{proof} 
\emph{Existence}. Let $X:=A^{\diamond_q,W}$. From parts (v) and (vi) of Theorem \ref{properties BT} it is clear that $X$ is a solution to \eqref{system 1}.\\
\emph{Uniqueness}. Any matrix $X$ satisfying conditions \eqref{system 1}, in particular satisfies   $\Ra(X)\subseteq\Ra((AW)^q)$ which is equivalent to $P_{(AW)^q}X=X$. Thus, from the condition  $P_{(AW)^q}X=(WAWP_{(AW)^q})^\dag$, we get $X=(WAWP_{(AW)^q})^\dag$, which gives  the conclusion.
\end{proof}

\begin{theorem} Let $A \in \Cm$, $0\neq W\in \Cn$, $k=\max\{\ind(AW), \ind(WA)\}$, and $q\in \mathbb{N}\cup \{0\}$. 
The unique matrix $X$ satisfying the conditions
\begin{equation}\label{system 2}
AWX=AW(WAWP_{(AW)^q})^\dag \quad \text{and} \quad \Ra(X)\subseteq\Ra(P_{(AW)^q}(WAW)^*)
\end{equation}
is given by $X=A^{\diamond_q,W}$.
\end{theorem}
\begin{proof} 
\emph{Existence}. Let $X:=A^{\diamond_q,W}$. By Definition \ref{def wqbt} and Theorem \ref{properties BT} (ii) it is clear that $X$ satisfies both conditions in \eqref{system 2}. \\
\emph{Uniqueness}. Let $X$ be an arbitrary matrix satisfying both conditions in \eqref{system 2}.
Since $\Ra((WAWP_{(AW)^q})^\dag)=\Ra(P_{(AW)^q}(WAW)^*)$, the second condition in \eqref{system 2} implies  $X=(WAWP_{(AW)^q})^\dag Z$ for some matrix $Z$. Now, from Lemma \ref{lemma 1} and the first equation in \eqref{system 2}  we obtain
\begin{eqnarray*}
X &=& (WAWP_{(AW)^q})^\dag Z \\
&=& (WAWP_{(AW)^q})^\dag WAW P_{(AW)^q}(WAWP_{(AW)^q})^\dag Z\\
&=& (WAWP_{(AW)^q})^\dag W AW [(WAWP_{(AW)^q})^\dag Z] \\
  &=&(WAWP_{AW})^\dag W (AWX) \\ 
  &=&(WAWP_{AW})^\dag WAW(WAWP_{(AW)^q})^\dag \\
  &=& (WAWP_{(AW)^q})^\dag WAW P_{(AW)^q}(WAWP_{(AW)^q})^\dag \\
  &=& (WAWP_{(AW)^q})^\dag \\
  &=& A^{\diamond_q,W},
\end{eqnarray*}
which gives the uniqueness.
\end{proof}

A similar result can be obtained using the null space.  

\begin{theorem}  Let $A \in \Cm$, $0\neq W\in \Cn$, $k=\max\{\ind(AW), \ind(WA)\}$, and $q\in \mathbb{N}\cup \{0\}$. The unique matrix $X$ that satisfies both conditions
\begin{equation}\label{system 3}
XWA=(WAWP_{(AW)^q})^\dag W A \quad \text{and}\quad \Nu(P_{(AW)^q}(WAW)^*)\subseteq\Nu(X) 
\end{equation}
is given by $X=A^{\diamond,W}$.
\end{theorem}

As a consequence of above results we obtain some characterizations of the $q$-BT inverse of a square matrix. 

\begin{theorem} Let $A\in\Cnn$.  The following statements are equivalent:
\begin{enumerate}[{\rm (i)}]
\item $X$ is the $q$-BT inverse of $A$;
\item $XAX=X$, $AX=A(AP_{A^q})^\dag$, and $XA=(AP_{A^q})^\dag A$;
\item $P_{A^q}X=(AP_{A^q})^\dag$ and $\Ra(X)\subseteq \Ra(A^q)$;
\item $AX=A(AP_{A^q})^\dag$ and $\Ra(X)\subseteq \Ra(P_{A^q}A^*)$;
\item $XA=(AP_{A^q})^\dag A$ and $\Nu(P_{A^q}A^*)\subseteq \Nu(X)$.
\end{enumerate}
\end{theorem}
\section{Canonical form of the $W$-weighted $q$-BT inverse}
In \cite{FeLeTh1} the authors introduced a simultaneous unitary block upper triangularization of a pair of rectangular matrices, called the weighted core-EP decomposition of the pair $(A,W)$. More precisely,  we have the following result:
\begin{theorem}\label{decomposition}
Let $A\in\Cm$ and $0\neq W\in \Cn$ with  $k=\max\{\ind(AW), \ind(WA)\}$.
 Then there exist two unitary matrices
 $U \in \Cmm$, $V \in \Cnn$, two nonsingular matrices $A_1, W_1 \in \mathbb{C}^{t \times t}$, and two matrices $A_3 \in
\mathbb{C}^{(m-t)\times(n-t)}$ and $W_3 \in
\mathbb{C}^{(n-t)\times(m-t)}$ such that $A_3W_3$ and $W_3A_3$ are
nilpotent of indices $\ind(AW)$ and  $\ind(WA)$,
respectively, with
\begin{equation} \label{weighted decomposition}
A = U\left[\begin{array}{cc}
A_1 & A_2 \\
0 & A_3
\end{array}\right]V^* \quad \text{and} \quad
W = V\left[\begin{array}{cc}
W_1 & W_2 \\
0 & W_3
\end{array}\right]U^*.
\end{equation}
\end{theorem}

The following lemma allows us to find the Moore-Penrose inverse of a partitioned matrix with some of its diagonal block nonsingular.

\begin{lemma}\cite{FeThTo}\label{MP triangular} 
Let
$A=U\begin{bmatrix}
        A_1 & A_2 \\
        0 & A_3 \\
      \end{bmatrix}V^* \in \Cm$ be  such that  $A_1 \in \C^{t\times t}$ is nonsingular and $U \in {\mathbb C}^{m \times m}$ and $V \in {\mathbb C}^{n \times n}$ are unitary. Then
\begin{equation} \label{moore penrose representacion}
A^\dag = V\left[\begin{array}{cc}
A_1^*\Omega & -A_1^*\Omega A_2 A_3^\dagger\\
(I_{n-t}-Q_{A_3})A_2^*\Omega & A_3^\dagger-(I_{n-t}-Q_{A_3})A_2^*\Omega A_2 A_3^\dagger
\end{array}\right]U^*,
\end{equation}
where $\Omega=[A_1 A_1^*+ A_2(I_{n-t}-Q_{A_3})A_2^*]^{-1}$.
In consequence, 
\begin{equation} \label{PA}
P_A= U\left[\begin{array}{cc}
I_t & 0\\
0 & P_{A_3}
\end{array}\right]U^*.
\end{equation} 
\end{lemma}
Now, we present a representation for the $W$-weighted $q$-BT inverses by using the weighted core-EP decomposition.  

\begin{theorem}\label{weighted qBT representation} 
Let $A \in \Cm$, $0\neq W\in \Cn$, $k=\max\{\ind(AW), \ind(WA)\}$, and $q\in \mathbb{N}\cup \{0\}$. If $A$ and $W$ are written  as in  (\ref{weighted decomposition}), then  the $W$-weighted $q$-BT  inverse of $A$ is given by
\begin{equation}\label{canonical q-BT}
 A^{\diamond_q,W} 
= U\left[\begin{array}{cc}
 (W_1A_1W_1)^* \Omega_W & -(W_1A_1W_1)^* \Omega_W M A_3^{\diamond_q,W_3} \\
(P_{(A_3W_3)^q}-P_{A_3^{\diamond_q,W_3}})M^*\Omega_W & A_3^{\diamond_q,W_3}-(P_{(A_3W_3)^q}-P_{A_3^{\diamond_q,W_3}})M^*\Omega_W M A_3^{\diamond_q,W_3}
\end{array}\right]V^*,
\end{equation}
where 
$$
M:=W_1A_1W_2+W_1A_2W_3+W_2A_3W_3
$$ 
and
$$
\Omega_W:=[W_1A_1W_1 (W_1A_1W_1)^*+ M (P_{(A_3W_3)^q}-P_{A_3^{\diamond_q,W_3}})M^*]^{-1}.
$$
\end{theorem}
\begin{proof}
We assume that $A$ and $W$ are written as in (\ref{weighted decomposition}). Applying Theorem \ref{existence uniqueness}, we have $A^{\diamond_q,W} = (WAWP_{(AW)^q})^\dag$. It can be easily obtained that  
\[
WAW=V \left[\begin{array}{cc}
W_1A_1W_1 & W_1A_1W_2+(W_1A_2+W_2A_3)W_3 \\
0 & W_3A_3W_3
\end{array}
\right]U^*=V \left[\begin{array}{cc}
W_1A_1W_1 & M \\
0 & W_3A_3W_3
\end{array}
\right]U^*,
\]
where $M:= W_1A_1W_2+W_1A_2W_3+W_2A_3W_3$, and
\[
P_{(AW)^q} = U\left[\begin{array}{cc}
I_t & 0 \\
0 & P_{(A_3W_3)^q}
\end{array}\right]U^*.
\]
Thus, we have that
\begin{equation*}
A^{\diamond_q,W} = (WAWP_{(AW)^q})^\dag
= 
U\left[\begin{array}{cc}
W_1 A_1 W_1 & M P_{(A_3W_3)^q}\\
0 & W_3A_3W_3 P_{(A_3W_3)^q}
\end{array}\right]^\dag V^* = U \left[\begin{array}{cc}
B_1 & B_2\\
B_3 & B_4
\end{array}\right] V^*,
\end{equation*}
where we are considering the partition given by the blocks $B_1,B_2,B_3$ and $B_4$ having appropriate sizes induced by the central matrix in the previous step.
By Theorem \ref{decomposition}, $W_1A_1W_1$ is nonsingular. In order to determine the blocks $B_1$, $B_2$, $B_3$, and $B_4$ we will use Lemma \ref{MP triangular}.
Taking $Z:=P_{(A_3W_3)^q}(I_{n-t}-Q_{W_3A_3W_3P_{(A_3W_3)^q}})P_{(A_3W_3)^q}$, we get
\begin{equation}\label{omega}
\Omega_W=[W_1A_1W_1 (W_1A_1W_1)^*+MZM^*]^{-1}.
\end{equation}
Moreover, from Lemma \ref{lemma 1} and Theorem \ref{existence uniqueness}, it follows 
\begin{eqnarray}\label{auxiliar}
Z &= &
(P_{(A_3W_3)^q})^2-[P_{(A_3W_3)^q} (W_3A_3W_3P_{(A_3W_3)^q})^\dag] W_3A_3W_3(P_{(A_3W_3)^q})^2 \nonumber \\
&= & P_{(A_3W_3)^q}-[W_3A_3W_3P_{(A_3W_3)^q}]^\dag W_3A_3W_3P_{(A_3W_3)^q}  \nonumber \\ 
&= & P_{(A_3W_3)^q}-A_3^{\diamond_q,W_3} [A_3^{\diamond_q,W_3}]^\dag \nonumber \\ 
&=& P_{(A_3W_3)^q}-P_{A_3^{\diamond_q,W_3}}.
\end{eqnarray}
From \eqref{auxiliar} and \eqref{omega} we have
\[\Omega_W=[W_1A_1W_1 (W_1A_1W_1)^*+ M (P_{(A_3W_3)^q}-P_{A_3^{\diamond_q,W_3}})M^*]^{-1}.\]
Finally, 
\begin{eqnarray*}
B_1 &=&(W_1A_1W_1)^* \Omega_W, \\ 
B_2 &=&-(W_1A_1W_1)^* \Omega_W M P_{(A_3W_3)^q} (W_3A_3W_3P_{(A_3W_3)^q})^\dag =\\
& =&  -(W_1A_1W_1)^* \Omega_W M A_3^{\diamond_q,W_3},\\
B_3&=&(I_{n-t}-Q_{W_3A_3W_3P_{(A_3W_3)^q}})(MP_{(A_3W_3)^q})^*\Omega_W = \\
&=&(P_{(A_3W_3)^q}-P_{A_3^{\diamond_q,W_3}})M^*\Omega_W,\\
B_4 &=&A_3^{\diamond_q,W_3}- B_3 M P_{(A_3W_3)^q} (W_3A_3W_3P_{(A_3W_3)^q})^\dag= \\
& =& A_3^{\diamond_q,W_3}- (P_{(A_3W_3)^q}-P_{A_3^{\diamond_q,W_3}})M^*\Omega_W M A_3^{\diamond_q,W_3},
\end{eqnarray*}
which completes the proof.
\end{proof}

\begin{corollary} \label{coro coincides} Let $A \in \Cm$, $0\neq W\in \Cn$, and $k=\max\{\ind(AW), \ind(WA)\}$. If $A$ and $W$ are written  as in  (\ref{weighted decomposition}), then the $W$-weighted BT inverse of $A$ is given by 
\begin{equation}\label{weighted BT}
A^{\diamond_1,W}=A^{\diamond,W} 
= U\left[\begin{array}{cc}
 (W_1A_1W_1)^* \Omega_W & -(W_1A_1W_1)^* \Omega_W M  A_3^{\diamond,W_3} \\
(P_{A_3W_3}-P_{A_3^{\diamond,W_3}})M^*\Omega_W & A_3^{\diamond,W_3}-(P_{A_3W_3}-P_{A_3^{\diamond,W_3}})M^*\Omega_W M A_3^{\diamond,W_3}
\end{array}\right]V^*, 
\end{equation}
where 
\begin{eqnarray*}
M &=& W_1A_1W_2+(W_1A_2+W_2A_3)W_3, \quad \text{and} \\
\Omega_W &=& [W_1A_1W_1 (W_1A_1W_1)^*+ M(P_{A_3W_3}- P_{A_3^{\diamond,W_3}})M^*]^{-1},
\end{eqnarray*}
and the $W$-weighted core-EP inverse of $A$ is given by 
\begin{equation}\label{weighted core ep}
 A^{\diamond_q,W}=A^{\odagger,W}= U\left[\begin{array}{cc}
(W_1A_1W_1)^{-1} & 0 \\
0 & 0
\end{array}\right]V^*, \quad \text{ for } q\ge k.
\end{equation}
\end{corollary}
\begin{proof}
By Corollary \ref{corolario 1} we know that $A^{\diamond_q,W}=A^{\diamond,W} $ if $q=1$ and $A^{\diamond_q,W}=A^{\odagger,W} $ if  $q\ge k$. Clearly, if $q=1$, \eqref{canonical q-BT} reduces to the expression given in \eqref{weighted BT}. \\ On the other hand, if $q\ge k$  we obtain  $(A_3W_3)^q=0$. In fact, since $A_3W_3$ is nilpotent of index at most $k$, we have $P_{(A_3W_3)^q}=0$. Hence, 
$A_3^{\diamond_q,W_3}=(W_3A_3W_3P_{(A_3W_3)^q})^\dag=0$. Now, from \eqref{omega}, it follows that
$\Omega_W=[W_1A_1W_1 (W_1A_1W_1)^*]^{-1}$. In this way, \eqref{canonical q-BT} reduces to \eqref{weighted core ep}. 
\end{proof}

\begin{remark} When $k=\max\{\ind(AW), \ind(WA)\}=1$, the above representations  coincide with the $W$-weighted core inverse, that is, $A^{\diamond,W}=A^{\odagger,W}=A^{\core,W}$.
\end{remark}

If $A\in \Cnn$ has index $k$, by applying Theorem \ref{decomposition} with $m=n$ and $W=I_n$, we obtain the following canonical form of $A$   
\begin{equation} \label{core-EP decomposition}
A = U\left[\begin{array}{cc}
T & S \\
0 & N
\end{array}\right]U^*,
\end{equation}
where $U\in \Cnn$ is unitary, $T$ is nonsingular, $\ra(T)=\ra(A^k)$, and $N$ is nilpotent of index $k$.
This representation of $A$ is called the core-EP decomposition of $A$ \cite{Wang}.

By using \eqref{core-EP decomposition} we can give a canonical form for the $q$-BT inverse of a square matrix. 

\begin{corollary}\label{canonical qbt} 
Let $A \in \Cnn$, $k=\ind(A)$, and  $q\in \mathbb{N}\cup \{0\}$. If $A$ is written  as in  \eqref{core-EP decomposition}, then  the  $q$-BT  inverse of $A$ is given by
\begin{equation}\label{canonical q-BT}
 A^{\diamond_q} 
= U\left[\begin{array}{cc}
T^*\Delta & -T^*\Delta S N^{\diamond_q}\\
(P_N-P_{N^{\diamond_q}})S^*\Delta & N^{\diamond_q}-(P_N-P_{N^{\diamond_q}})S^*\Delta S N^{\diamond_q}
\end{array}\right]U^*,
\end{equation}
where $\Delta=(T T^*+ S(P_N-P_{N^{\diamond_q}})S^*)^{-1}$.
\end{corollary}

\begin{corollary}
Let $A \in \Cm$, $0\neq W\in \Cn$, and $k=\max\{\ind(AW), \ind(WA)\}$. If $A$ and $W$ are written  as in  (\ref{weighted decomposition}), then it results that
\[
(AW)^{\diamond_q}=U\left[\begin{array}{cc}
(A_1W_1)^*\Delta & -(A_1W_1)^*\Delta S (A_3W_3)^{\diamond_q}\\
(P_{A_3W_3}-P_{(A_3W_3)^{\diamond_q}})S^*\Delta & (A_3W_3)^{\diamond_q}-(P_{A_3W_3}-P_{(A_3W_3)^{\diamond_q}})S^*\Delta S (A_3W_3)^{\diamond_q}
\end{array}\right]U^*,
\]
with  $\Delta=(A_1W_1(A_1W_1)^*+ S(P_{A_3W_3}-P_{(A_3W_3)^{\diamond_q}})S^*)^{-1}$ and $S=A_1W_2+A_2W_3$,
and 
\[
(WA)^{\diamond_q}=U\left[\begin{array}{cc}
(W_1A_1)^*\Delta & -(W_1A_1)^*\Delta S (W_3A_3)^{\diamond_q}\\
(P_{W_3A_3}-P_{(W_3A_3)^{\diamond_q}})S^*\Delta & (W_3A_3)^{\diamond_q}-(P_{W_3A_3}-P_{(W_3A_3)^{\diamond_q}})S^*\Delta S (W_3A_3)^{\diamond_q}
\end{array}\right]U^*,
\]
with  $\Delta=(W_1A_1(W_1A_1)^*+ S(P_{W_3A_3}-P_{(W_3A_3)^{\diamond_q}})S^*)^{-1}$ and $S=W_1A_2+W_2A_3$.
\end{corollary}
\begin{proof} 
From Theorem \ref{decomposition} we obtain
\begin{equation}\label{wa}
AW= U\left[\begin{array}{cc}
A_1W_1 & A_1W_2 + A_2 W_3\\
0 & A_3W_3
\end{array}\right]U^*,
\end{equation}
where $U$ is unitary, $A_1W_1$ is nonsingular, and $A_3W_3$ is nilpotent of index $\ind(AW)$.\\
Clearly, \eqref{wa} is a core-EP decomposition of $AW$. Thus, the 
the expression for $(AW)^{\diamond_q}$ follows from Corollary \ref{canonical qbt}. 
\\
The expression for 
$(WA)^{\diamond_q}$ can be found 
in a  similar way. 
\end{proof}

We recall that the $W$-weighted Drazin inverse and the $W$-weighted core-EP inverse of $A$ satisfy the interesting identities $A^{d,W}=[(AW)^d]^2A=A[(WA)^d]^2$ and $A^{\odagger,W}=A[(WA)^{\odagger}]^2$, from \eqref{properties w drazin} and \eqref{properties w core ep}, respectively.  \\ However, these equalities do not remain valid for the $W$-weighted $q$-BT inverse whenever $1\le q <k=\max\{{\rm Ind}(AW), {\rm Ind}(WA)\}$, as we can check with the following example.

\begin{example} Let
\[A=\left[\begin{array}{ccc}
1 &  1 & 0 \\
0 &  1 & 0 \\
0 &  0 & 1 \\
0 &  0 & 0
\end{array}\right] \quad \text{and}  \quad W=\left[\begin{array}{cccc}
1 &  1 & 0 & 0 \\
0 &  0 & 1 & 0 \\
0 &  0 & 0 & 1
\end{array}\right].\]
Note that $k=\max\{{\rm Ind}(AW), {\rm Ind}(WA)\}=\max\{3,2\}=3$. For  $1\le q<3$ we obtain
\[A^{\diamond_1,W}=
\left[\begin{array}{ccc}
\frac{1}{6} &  0 & 0 \\
\frac{1}{6} &  0 & 0 \\
\frac{1}{3} &  0 & 0 \\
0 &  0 & 0
\end{array}\right],\quad  
[(AW)^{\diamond_1}]^2A=\left[\begin{array}{ccc}
0 &  0 & 0 \\
0 &  0 & 0 \\
\frac{1}{2} &  0 & 0 \\
0 &  0 & 0
\end{array}\right] \quad \text{and}\quad A[(WA)^{\diamond_1}]^2=
\left[\begin{array}{ccc}
\frac{3}{25} &  0 & 0 \\
\frac{2}{25} &  0 & 0 \\
0 &  0 & 0 \\
0 &  0 & 0
\end{array}\right],\]
\[A^{\diamond_2,W}=
\left[\begin{array}{ccc}
\frac{1}{2} &  0 & 0 \\
\frac{1}{2} &  0 & 0 \\
0 &  0 & 0 \\
0 &  0 & 0
\end{array}\right],\quad  
[(AW)^{\diamond_2}]^2A=
\left[\begin{array}{ccc}
\frac{1}{4} &  \frac{1}{4} & 0 \\
\frac{1}{4} &  \frac{1}{4} & 0 \\
0 &  0 & 0 \\
0 &  0 & 0
\end{array}\right] \quad \text{and}\quad A[(WA)^{\diamond_2}]^2=\left[\begin{array}{ccc}
1 &  0 & 0 \\
0 &  0 & 0 \\
0 &  0 & 0 \\
0 &  0 & 0
\end{array}\right].\]
\end{example}

\begin{remark} If we take $q=3$ in the above example (i.e., $q=k=3$), from Corollary \ref{coro coincides} we have that $A^{\diamond_3,W}=A^{\odagger,W}$. Thus, from \eqref{properties w core ep} we obtain $A^{\diamond_3,W}=A[(WA)^{\diamond_3}]^2$,  which can be verified in the example given above, that is, 
\[A^{\diamond_3,W}=
\left[\begin{array}{ccc}
1 &  0 & 0 \\
0 &  0 & 0 \\
0 &  0 & 0 \\
0 &  0 & 0
\end{array}\right],\quad  
[(AW)^{\diamond_3}]^2A=
\left[\begin{array}{ccc}
1 &  1 & 0 \\
0 &  0 & 0 \\
0 &  0 & 0 \\
0 &  0 & 0
\end{array}\right], \quad \text{and}\quad A[(WA)^{\diamond_3}]^2=\left[\begin{array}{ccc}
1 &  0 & 0 \\
0 &  0 & 0 \\
0 &  0 & 0 \\
0 &  0 & 0
\end{array}\right].\]
\end{remark}

\section*{Author Contribution} 

All authors contributed equally to the writing of this paper. All authors read and approved the final manuscript.

\section*{Data Availability} 

No data was used.

\section*{Funding} 

The first and third author are partially supported by Universidad Nacional de R\'{\i}o Cuarto (Grant PPI 18/C559) and CONICET (Grant PIBAA 28720210100658CO). The second author was partially supported by Universidad Nacional de La Pampa, Facultad de Ingenier\'ia (Grant Resol. Nro. 135/19) and Ministerio de Econom\'{\i}a, Industria y Competitividad (Spain) [Grant Red de Excelencia RED2022-134176-T].

\section*{Declarations} 
 
\noindent {\bf Conflict of interest} The authors have no conflicts of interest.

\end{document}